\documentclass[12pt,a4paper]{article}
\usepackage{amsmath,amsfonts,amssymb,amsthm,amscd}
\newtheorem{thm}{Theorem}
\newtheorem{lem}[thm]{Lemma}
\newtheorem{prop}[thm]{Proposition}
\newtheorem{defn}[thm]{Definition}
\newtheorem{cor}[thm]{Corollary}
\newtheorem{rem}[thm]{Remark}
\newtheorem{notations}[thm]{Notation}
\newtheorem{exa}[thm]{Example}

\begin{document}

\newfont{\goth}{eufm10 scaled \magstep1}
\def\ga{\mbox{\goth a}}
\def\gp{\mbox{\goth p}}
\def\gc{\mbox{\goth c}}
\def\gg{\mbox{\goth g}}
\def\ge{\mbox{\goth e}}
\def\gh{{\mbox{\goth h}}}
\def\gk{\mbox{\goth k}}
\def\gl{\mbox{\goth l}}
\def\gf{\mbox{\goth f}}
\def\gm{\mbox{\goth m}}
\def\gn{\mbox{\goth n}}
\def\gq{\mbox{\goth q}}
\def\gr{\mbox{\goth r}}
\def\gs{\mbox{\goth s}}
\def\gt{\mbox{\goth t}}
\def\gu{\mbox{\goth u}}
\def\gw{\mbox{\goth w}}
\def\gsl{\mbox{\goth sl}}
\def\gsp{\mbox{\goth sp}}
\def\r{\mbox{\goth r}}
\def\gso{\mbox{\goth so}}
\def\gsu{\mbox{\goth su}}
\newcommand\Id{\mathrm{Id}}
\newcommand\Ad{\mathrm{Ad}}
\newcommand\ad{\mathrm{ad}}

\author{Liana David}

\title{On cotangent manifolds, complex structures and generalized
geometry}

\maketitle

\textbf{Abstract:} We develop various properties of  symmetric
generalized complex structures (in connection with their
holomorphic space and $B$-field transformations), which are
analogous to the well-known results of Gualtieri \cite{gualtieri}
on skew-symmetric generalized complex structures. Given a
symmetric or skew-symmetric generalized complex structure
$\mathcal J$ and a linear connection $D$ on a manifold $M$, we construct
an almost complex structure $J^{\mathcal J , D}$ on the cotangent
manifold $T^{*}M$ and  we study its integrability.  For $\mathcal
J$ skew-symmetric, we relate the Courant integrability of
$\mathcal J$ with the integrability of $J^{\mathcal J , D}$. 
We consider in detail the case when $M=G$ is a Lie group and $\mathcal J$, $D$ are 
left-invariant. We also
show that our approach unifies and generalizes various results
from special complex geometry.

\section{Introduction}

{\bf Motivation.} The starting point of this note is a result proved in \cite{alek},
which states that the cotangent manifold of a special symplectic
manifold $(M, J, \nabla , \omega )$ inherits, under some
additional conditions, a hyper-K\"{a}hler structure. Recall that a
manifold $M$ with a complex structure $J$, a flat, torsion-free connection
$\nabla$ and a symplectic form $\omega$ is special symplectic if
$d^{\nabla}J=0$ (i.e. $\nabla_{X}(J)(Y) = \nabla_{Y}(J)(X)$, for
any $X, Y\in TM$) and $\nabla \omega =0.$ The connection $\nabla$,
acting on the cotangent bundle $\pi : T^{*}M\rightarrow M$,
induces a decomposition
\begin{equation}\label{deco}
T(T^{*}M) = H^{\nabla}\oplus\pi^{*}T^{*}M= \pi^{*}(TM\oplus
T^{*}M)
\end{equation}
into horizontal and vertical subbundles. Assume now that the
$(1,1)$-part of $\omega$ (with respect to $J$) is non-degenerate
and satisfies $\nabla \omega^{1,1}=0$. Under these additional
conditions, the hyper-K\"{a}hler structure on $T^{*}M$ mentioned
above is given, by means of (\ref{deco}), by (the pull-back of)
$$
J_{1}:= \left(
\begin{tabular}{cc}
$J$ & $0$\\
$0$ & $J^{*}$\\
\end{tabular}\right) , \quad J_{2}:=
\left(\begin{tabular}{cc}
$0$ & $-(\omega^{1,1})^{-1}$\\
$\omega^{1,1}$ & $0$\\
\end{tabular}\right), \quad
g:=\left(
\begin{tabular}{cc}
$g^{1,1}$ & $0$\\
$0$ & $(g^{1,1})^{-1}$\\
\end{tabular}\right) ,
$$
where $g^{1,1}:=\omega^{1,1}(J\cdot , \cdot ).$ A key fact in the
proof that $(J_{1}, J_{2}, g)$ is hyper-K\"{a}hler is the
integrability of $J_{1}$ and $J_{2}.$ The integrability of $J_{2}$
follows from a local argument, which uses $\nabla$-flat
coordinates and $\nabla \omega^{1,1}=0$. For the integrability of
$J_{1}$, one notices, using the special complex condition
$d^{\nabla}J=0$, that $H^{\nabla} \subset T(T^{*}M)$ is invariant
with respect to the canonical complex structure $J_{\mathrm{can}}$
of $T^{*}{M}$ induced by $J$. Hence, $J_{1}$ coincides with
$J_{\mathrm{can}}$ and is integrable. These arguments were developed
in \cite{alek}.

With special geometry as a motivation, in this note we consider
the following setting: a manifold $M$ with a linear connection $D$
and a smooth field of endomorphisms $\mathcal J$ of the
generalized tangent bundle $\mathbb{T}M:= TM\oplus T^{*}M$, such
that $\mathcal J^{2} = -\mathrm{Id}$. Following \cite{nan} (rather
than the usual terminology from generalized geometry), we call
$\mathcal J$ a generalized complex structure. Motivated by $J_{1}$
and $J_{2}$ above, we assume that $\mathcal J$ is symmetric or
skew-symmetric with respect to the canonical metric of neutral
signature of $\mathbb{T}M$. From $D$ and $\mathcal J$ we construct
an almost complex structure $J^{\mathcal J , D}$ on the cotangent
manifold $T^{*}M$ and we study its integrability. This provides a
new insight, from the generalized complex geometry point of view,
on the above arguments from \cite{alek}. Along the way, we prove
various properties we need on symmetric generalized complex
structures. The relation with the Courant integrability is also
discussed. As a main application, we construct a large class of
complex structures on cotangent manifolds of real semisimple Lie
groups.

In the remaining part of the introduction we describe
in detail the results and the structure of the paper.\\

{\bf Structure of the paper.} In Section \ref{basic} we prove
basic facts we need from generalized geometry. While
skew-symmetric generalized complex structures are well-known (see
e.g. \cite{gualtieri} for basic facts), the symmetric ones do not
seem to appear in the literature. We begin by studying symmetric
generalized complex structures on (real) vector spaces. We find
the general form of their holomorphic space (see Proposition
\ref{holo}) and we show that any symmetric generalized complex
structure on a vector space is, modulo a $B$-field transformation,
the direct sum of one determined by a complex structure and
another determined by a pseudo-Euclidian metric (see Example
\ref{example} and Theorem \ref{b-field}). Therefore, there is an
obvious analogy with the theory of skew-symmetric generalized
complex structures developed by Gualtieri in \cite{gualtieri}. We
discuss this analogy in Subsection \ref{analogies}. For our
purposes it is particularly relevant the common description of the
holomorphic space $L^{\tau}(E, \alpha )$ of a symmetric or,
respectively, skew-symmetric generalized complex structure on a
vector space $V$, in terms of a complex subspace $E \subset
V^{\mathbb{C}}$ and a skew-Hermitian, respectively skew-symmetric
$2$-form $\alpha$ on $E$, satisfying some additional conditions
(see Corollary \ref{common}). These results extend pointwise to
manifolds (see Subsection \ref{integr}). Despite the above
analogies, there is an important difference between symmetric and
skew-symmetric generalized complex structures on manifolds: unlike
the skew-symmetric ones, the symmetric generalized complex
structures are never Courant integrable (see Lemma
\ref{non-integrable}).

In Section \ref{main-section} we consider a manifold $M$ together
with a connection $D$ and a symmetric or skew-symmetric
generalized complex structure $\mathcal J .$ Using $D$ and
$\mathcal J$ we define an almost complex structure $J^{\mathcal J
, D}$ on $T^{*}M$ (see Definition \ref{a-c}) and we discuss its
integrability. It turns out that the integrability of $J^{\mathcal
J , D}$ imposes obstructions on the curvature of $D$ and the data
$(E, \alpha )$ which defines the holomorphic bundle $L=
L^{\tau}(E, \alpha )$ of $\mathcal J .$ In particular, the complex
subbundle $E\subset T^{\mathbb{C}}M$ must be involutive and
$\alpha$ must satisfy a differential equation involving $D$ (see
Theorem \ref{main-thm}). As a straightforward application of
Theorem \ref{main-thm},  we relate the Courant integrability of a
skew-symmetric generalized complex structure $\mathcal J$,  with the
integrability of $J^{\mathcal J, D}$ (see Corollary
\ref{remarca}).  In particular, we deduce that a
left-invariant, skew-symmetric, generalized complex structure $\mathcal J$ on a Lie
group $G$ is Courant integrable, if and only if  the almost
complex structure $J^{\mathcal J , D^{c}}$ on $T^{*}G$ is
integrable, where $D^{c}$ is the left-invariant connection which
on left-invariant vector fields is the Lie bracket (see Example
\ref{Courant-Lie}). A systematic description of 
Courant integrable, left-invariant, skew-symmetric generalized complex structures on real semisimple
Lie groups was developed in \cite{liana}. This is the motivation for our treatment from the next section.

Section \ref{complex} is  devoted to applications of Theorem
\ref{main-thm} to Lie groups. Our main goal here is to describe a
large class of left-invariant symmetric (rather than skew-symmetric) generalized
complex structures $\mathcal J$  on a semisimple Lie group, which, together with a suitably
chosen left-invariant connection $D^{0}$, 
determine an integrable complex structure  $J^{\mathcal J , D^{0}}$ on the cotangent group
 (The connection $D^{0}$ plays the role of $D^{c}$ above). In the first part of Section
\ref{complex}, intended to fix notations, we briefly recall the
basic facts we need on the structure theory of semisimple Lie algebras. We follow closely
\cite{knapp}, Chapter VI. In Subsection \ref{semi1} we develop an
infinitesimal description, in terms of the so-called admissible
triples $(\gk , \mathcal D ,\epsilon )$, of pairs $(\mathcal J ,
D)$ formed by a left-invariant (symmetric or skew-symmetric)
generalized  complex structure $\mathcal J$ and a left-invariant
connection $D$ on a (not necessarily semisimple) Lie group $G$, such that the associated almost
complex structure $J^{\mathcal J , D}$ on $T^{*}G$ is integrable
(see Definition \ref{adm} and Proposition \ref{cons}). In this
description, the pair $(\gk , \epsilon )$ defines the fiber
$L^{\tau}(\gk , \epsilon )$ at $e\in G$ of the holomorphic bundle
of $\mathcal J$ and $\mathcal D $ is the restriction of $D$ to the
space of left-invariant vector fields. The notion of admissible
triple generalizes the notion of admissible pair, defined in
\cite{liana} to encode the Courant integrability of left-invariant
skew-symmetric generalized complex structures on Lie groups.  When
$G$ is semisimple, we define the notion of regularity for the
structures involved (see Definition \ref{adm}); in the above
notation, this means that $\gk$ is a regular subalgebra of
$\gg^{\mathbb{C}}$, normalized by a maximally compact Cartan subalgebra of $\gg$. The preferred connection
$D^{0}$ is introduced in Definition \ref{def-d-0}  and our motivation
for its choice is explained before Lemma \ref{connection}. Our main result in
this section is Theorem \ref{main-applic}, which provides a
description (in terms of admissible triples) of regular symmetric
generalized complex structures $\mathcal J$ on $G$, which,
together with the connection $D^{0}$, determine an (integrable)
complex structure on $T^{*}G$. The description from Theorem
\ref{main-applic} requires further clarifications: one needs to
construct the constants $\{ \nu_{\alpha}, \, \alpha \in
R_{0}^{\mathrm{sym}}\}$, which are subject to conditions
(\ref{ad-nu}), (\ref{suplimentara}) and to study the
non-degeneracy of the (symmetric) bilinear form $g_{\Delta}$. A
method to construct the $\nu_{\alpha}$'s is provided by Lemma
\ref{nu}. When the root system $R_{0}$ of the regular subalgebra
$\gk$ is not only $\sigma$-parabolic, as required by Theorem
\ref{main-applic}, but $\sigma$-positive (see Definition \ref{def-sigma}),
the non-degeneracy of
$g_{\Delta}$ is straightforward (see Remark \ref{outer}) and we
obtain, on any semisimple Lie group $G$, a large class of regular
symmetric generalized complex structures $\mathcal J$, such that
$J^{\mathcal J , D^{0}}$ is integrable. In the special case when
$G$ is of inner type, the root system $R_{0}$ of $\gk$ is always a
positive root system and we obtain a full explicit description of
all regular symmetric generalized complex structures $\mathcal J$,
such that $J^{\mathcal J , D^{0}}$ is integrable (see Theorem
\ref{inner-descr}).

In Section \ref{special-geom} we use Theorem \ref{main-thm} in
order to reobtain and generalize various well-known results from
special complex geometry, with emphasis on those from \cite{alek},
already mentioned at the beginning of this introduction.\\

{\bf Acknowledgements.} Part of this work was completed during a
Humboldt Research Fellowship at the University of Mannheim
(Germany). Hospitality at the University of Mannheim and financial
support from the Humboldt Foundation are greatly acknowledged.
Partial supported from a CNCS-UEFISCDI grant, project no.
PN-II-ID-PCE-2011-3-0362 is also acknowledged.

\section{Symmetric generalized complex structures}\label{basic}

In this section we study symmetric generalized complex structures.
Subsections \ref{linear} and \ref{analogies} are algebraic, while
in Subsection \ref{integr} we discuss the Courant integrability.

\subsection{Linear symmetric generalized complex
structures}\label{linear}

Let $V$ be a real vector space. We denote by
\begin{equation}\label{g-can}
g_{\mathrm{can}} (X + \xi, Y + \eta ) = \frac{1}{2}\left( \xi (Y)
+\eta (X)\right) , \quad X+\xi , Y+ \eta \in V\oplus V^{*}
\end{equation}
the canonical pseudo-Euclidian metric of neutral signature on
$V\oplus V^{*}.$

\begin{defn}
A (symmetric, respectively skew-symmetric) generalized complex
structure on $V$ is a (symmetric, respectively skew-symmetric with
respect to $g_{\mathrm{can}}$) endomorphism $\mathcal J \in
\mathrm{End}( V \oplus V^{*})$, such that $\mathcal J^{2} = -
\mathrm{Id}$.

\end{defn}

\begin{rem}{\rm In the classical
terminology of generalized geometry (see e.g.
\cite{gualtieri,hitchin}),  a generalized complex structure is, by
definition, skew-symmetric. In this note we prefer the language of
\cite{nan}, where generalized complex structures are not assumed,
a priori, to be compatible in any way with $g_{\mathrm{can}}.$}
\end{rem}

In the following proposition we describe the holomorphic space of
symmetric generalized complex structures. Before we need to
introduce a notation which will be used along the paper.

\begin{notations}{\rm For a complex  subspace $E\subset
V^{\mathbb{C}}$, we denote by $\bar{E}$ the image of $E$ through
the antilinear conjugation $V^{\mathbb{C}}\ni X \rightarrow
\bar{X}\in V^{\mathbb{C}}$ with respect to the real form $V$ of
$V^{\mathbb{C}}$. In particular, $\bar{E}$ is a complex subspace
of $V^{\mathbb{C}}$ (not to be confused with the conjugate vector
space of $E$).}
\end{notations}

\begin{prop}\label{holo} A complex subspace $L$ of
$(V\oplus V^{*})^{\mathbb{C}}$ is the holomorphic space of a
symmetric generalized complex structure on $V$ if and only if it
is of the form
\begin{equation}\label{charact}
L= L^{-}(E, \alpha ):= \{ X+\xi \in E\oplus (V^{\mathbb{C}})^{*},
\quad \xi\vert_{\bar{E}} = i_{X} \alpha \} ,
\end{equation}
where $E$ is any complex subspace of $V^{\mathbb{C}}$, such that
$E+ \bar{E}= V^{\mathbb{C}}$, and $\alpha\in E^{*}\otimes
\bar{E}^{*}$ is any complex bilinear form satisfying the following
two conditions:\

i) it is skew-Hermitian, i.e.
\begin{equation}\label{sym-bar}
\alpha (X, \bar{Y}) + \overline{\alpha (Y, \bar{X})} = 0, \quad
\forall X,Y\in E;
\end{equation}

ii)  $\mathrm{Im}(\alpha \vert_{\Delta})$ is non-degenerate. Here
$\Delta\subset V$ is the real part of $E\cap \bar{E}$, i.e.
$\Delta^{\mathbb{C}} = E \cap \bar{E}$.
\end{prop}

\begin{proof}
Let $\mathcal J$ be a symmetric generalized complex structure on
$V$, with holomorphic space $L.$ Thus $L$ is a complex subspace of
$(V\oplus V^{*})^{\mathbb{C}}$, with $L\oplus \bar{L}= (V\oplus
V^{*})^{\mathbb{C}}$, and $L$ is $g_{\mathrm{can}}$-orthogonal to
$\bar{L}$ (from the symmetry of $\mathcal J$). We denote by
$$
\pi_{1} : (V\oplus V^{*})^{\mathbb{C}}\rightarrow V^{\mathbb{C}},
\quad \pi_{2} : (V\oplus V^{*})^{\mathbb{C}}\rightarrow
(V^{\mathbb{C}})^{*}
$$
the natural projections. We define $E:= \pi_{1} (L)$ and we let
\begin{equation}\label{e}
\alpha : E \rightarrow \bar{E}^{*}, \quad \alpha (X) :=
\pi_{2}\circ (\pi_{1}\vert_{L})^{-1} (X)\vert_{\bar{E}}.
\end{equation}
We claim that $\alpha\in E^{*}\otimes \bar{E}^{*}$ is well
defined. To prove this claim, we use
\begin{equation}\label{ortho-bar}
\xi (\bar{Y}) + \bar{\eta} (X) = 0, \quad \forall X+\xi , Y+\eta
\in L,
\end{equation}
(which holds because $L$ is $g_{\mathrm{can}}$-orthogonal to $\bar{L}$). Thus,
if $X+\xi_{1}, X+\xi_{2}\in (\pi_{1}\vert_{L})^{-1}(X)$, i.e.
$X+\xi_{1},X+\xi_{2}\in L$, then, from (\ref{ortho-bar}),
$\xi_{1}=\xi_{2}$ on $\bar{E}$ and we obtain that $\alpha$ is
well-defined, as required. From the very definition of $\alpha$,
$L \subset L^{-}(E, \alpha )$ and, for dimension reasons, we
deduce that $L= L^{-} (E, \alpha ).$ Since  $L$ is
$g_{\mathrm{can}}$-orthogonal to $\bar{L}$, $\alpha$ is
skew-Hermitian. Moreover, $L\oplus \bar{L}= (V\oplus
V^{*})^{\mathbb{C}}$ implies that $E+ \bar{E} = V^{\mathbb{C}}$.
We now claim that $L\cap \bar{L} = \{ 0\}$ implies that
$\mathrm{Im} (\alpha\vert_{\Delta})$ is non-degenerate. To prove
this claim, we assume, by absurd, that there is $X\neq 0$ in the
kernel of $\mathrm{Im} (\alpha\vert_{\Delta})$. Define $\xi \in
(V^{\mathbb{C}})^{*}$ by
$$
\xi (Z)=\overline{\alpha (X, \bar{Z})}, \quad \xi (\bar{Z}) =
\alpha (X, \bar{Z}), \quad \forall Z\in E.
$$
Using that $X\in \mathrm{Ker}(\mathrm{Im}
(\alpha\vert_{\Delta}))$, one can check that $\xi$ is well-defined
and $X+\xi \in L \cap \bar{L}$, which is a contradiction. We
proved that the holomorphic space $L$ of $\mathcal J$ is of the
required form.

Conversely, it may be shown that any subspace $E\subset
V^{\mathbb{C}}$, with $E+ \bar{E} = V^{\mathbb{C}}$, and
skew-Hermitian form $\alpha\in E^{*}\otimes \bar{E}^{*}$ with the
non-degeneracy property {\it ii)}, define, by (\ref{charact}), the
holomorphic space of a symmetric generalized complex structure on
$V$.
\end{proof}

\begin{cor}  Let $\mathcal J$ be a symmetric generalized complex
structure on $V$, with holomorphic space $L^{-}(E, \alpha )$. Then
 $\mathrm{Re}(\alpha\vert_{\Delta} )$ is a $2$-form  and
$\mathrm{Im}(\alpha\vert_{\Delta} )$ is a pseudo-Euclidian
metric on
$\Delta$ (the real part of $E\cap \bar{E}$).\end{cor}

\begin{proof} Straightforward, from (\ref{sym-bar}) and the non-degeneracy of
$\mathrm{Im}(\alpha\vert_{\Delta})$.
\end{proof}

The second example below shows that symmetric generalized complex
structures exist on vector spaces of arbitrary dimension.

\begin{exa}\label{example}{\rm i) A complex structure $J$ on $V$ defines a
symmetric generalized complex structure
$$
{\mathcal J}:= \left(\begin{tabular}{cc} $J$ & $0$\\
$0$ & $J^{*}$

\end{tabular}\right) ,
$$
where $J^{*}\xi := \xi \circ J $, for any $\xi \in V^{*}.$   Its
holomorphic space is $L^{-}(V^{1,0}, 0)= V^{1,0}\oplus\mathrm{Ann}
(V^{0,1})$, where $V^{1,0}$ is the
holomorphic space of $J$.\\

ii) A pseudo-Euclidian metric, seen as a map $g: V \rightarrow
V^{*}$, defines a symmetric generalized complex structure
$$
{\mathcal J}:= \left(\begin{tabular}{cc} $0$ & $g^{-1}$\\
$-g$ & $0$
\end{tabular}\right) .
$$
Its holomorphic space is $L^{-}(V^{\mathbb{C}}, i
g^{\mathbb{C}})$, where $g^{\mathbb{C}}\in (V^{\mathbb{C}}\otimes
V^{\mathbb{C}})^{*}$ is
the complex linear extension of $g$.\\

iii) If $\mathcal J$ is a symmetric generalized complex structure,
then so is its $B$-field transformation $ \mathrm{exp}(B)\cdot
{\mathcal J}:= \mathrm{exp}(B) \circ {\mathcal J} \circ
\mathrm{exp}(-B)$, where $B\in \Lambda^{2}(V^{*})$ and the
$B$-field action is defined by
$$
\mathrm{exp}(B): V\oplus V^{*} \rightarrow V\oplus V^{*}, \quad
X+\xi \rightarrow X + i_{X}B +\xi .
$$
If $L^{-}(E, \alpha )$ is the holomorphic space of $\mathcal J$,
then $L^{-}(E, \alpha +B^{\mathbb{C}}\vert_{E\otimes \bar{E}})$ is
the holomorphic space of $\mathrm{exp}(B)\cdot {\mathcal J}$,
where $B^{\mathbb{C}}\in \Lambda^{2} (V^{\mathbb{C}})^{*}$ is the
complex linear extension of $B$.}
\end{exa}

In following theorem we show that any symmetric generalized
complex structure can be (non-canonically) obtained from a complex
structure, a pseudo-Euclidian metric and a $B$-field
transformation.

\begin{thm}\label{b-field}
Any symmetric generalized complex structure on a vector space $V$
is a $B$-field transformation of the direct sum of one determined
by a complex structure and another determined by a
pseudo-Euclidian metric (as in Example \ref{example}).
\end{thm}

\begin{proof}
Let $\mathcal J\in \mathrm{End}(V\oplus V^{*})$ be a symmetric
generalized complex structure, with holomorphic space $L= L^{-}(E,
\alpha )$. Let $\Delta$ be the real
part of $E\cap \bar{E}$ (i.e. $\Delta \subset V$ and
$\Delta^{\mathbb{C}} = E\cap \bar{E}$) and $N$ a complement of
$\Delta$ in $V$. Thus
\begin{equation}\label{descompuneri}
V = \Delta \oplus N, \quad E =\Delta^{\mathbb{C}}\oplus (E \cap
N^{\mathbb{C}}), \quad \bar{E}=\Delta^{\mathbb{C}}\oplus (\bar{E}
\cap N^{\mathbb{C}}).
\end{equation}
We notice that $\Delta$ comes with pseudo-Euclidian metric, namely
$g_{\Delta}:= \mathrm{Im}(\alpha\vert_{\Delta} )$, and $N$ with a complex
structure $J^{N}$, with holomorphic space $E\cap N^{\mathbb{C}} $
(and anti-holomorphic space $\bar{E}\cap N^{\mathbb{C}}$). We
claim that there is $B\in \Lambda^{2}(V^{*})$ such that (as vector
spaces with symmetric generalized complex structures)
\begin{equation}\label{search-B}
(V,\mathrm{exp}(B)\cdot \mathcal J )=(\Delta ,
g_{\Delta} )\oplus (N, J^{N}),
\end{equation}
or, in terms of their holomorphic spaces,
\begin{equation}\label{search-B-0}
L^{-} (E, \alpha + B^{\mathbb{C}}\vert_{E\otimes \bar{E}}) =
L^{-}(\Delta^{\mathbb{C}}, i
(g_{\Delta})^{\mathbb{C}}) \oplus \left( E\cap
N^{\mathbb{C}} \oplus \mathrm{Ann} (\bar{E}\cap
N^{\mathbb{C}})\right).
\end{equation}
From the second and third relation
(\ref{descompuneri}), we obtain that (\ref{search-B-0}) holds if
and only if, for any $X\in E$, the covector $i_{X} (\alpha +
B^{\mathbb{C}})\in \bar{E}^{*}$ is given by
\begin{equation}\label{conditia2}
i_{X} (\alpha + B^{\mathbb{C}}) \vert_{\Delta^{\mathbb{C}}} =
i(g_{\Delta})^{\mathbb{C}}
(\mathrm{pr}_{\Delta^{\mathbb{C}}}(X), \cdot ), \quad i_{X}
(\alpha + B^{\mathbb{C}})\vert_{\bar{E}\cap N^{\mathbb{C}}}=0,
\end{equation}
where $\mathrm{pr}_{\Delta^{\mathbb{C}}}:
V^{\mathbb{C}}\rightarrow \Delta^{\mathbb{C}}$ and
$\mathrm{pr}_{N^{\mathbb{C}}}:V^{\mathbb{C}} \rightarrow
N^{\mathbb{C}}$ are the natural projections determined by the
decomposition $V^{\mathbb{C}}= \Delta^{\mathbb{C}}\oplus
N^{\mathbb{C}}.$ Moreover, it is easy to see that
(\ref{conditia2}) is equivalent to
\begin{equation}\label{cerinta}
(\mathrm{Re} (\alpha )+B) \vert_{\Delta\otimes \Delta} = 0, \quad
(\alpha + B^{\mathbb{C}} )\vert_{(E\cap N^{\mathbb{C}})\otimes
\Delta^{\mathbb{C}}} = 0, \quad (\alpha + B^{\mathbb{C}}
)\vert_{E\otimes (\bar{E}\cap N^{\mathbb{C}})}= 0.
\end{equation}
Hence, we are looking for a (real) $2$-form $B\in
\Lambda^{2}(V^{*})$ such that (\ref{cerinta}) is satisfied. In
order to define $B$, we use $V = \Delta \oplus N$ and
$N^{\mathbb{C}} = (E \cap N^{\mathbb{C}}) \oplus ( \bar{E}\cap
N^{\mathbb{C}}).$ Then, for any $X ,Y\in \Delta$ and $Z,W\in N$,
let
$$
B (X, Y):= - \mathrm{Re}(\alpha )(X,Y), \quad B(Z,W) := -2
\mathrm{Re}(\alpha )(z, \bar{w})
$$
and
$$
 B(X, Z) = - B(Z, X):= 2 \mathrm{Re} (\alpha) (z,X),
$$
where $z,w\in E\cap N^{\mathbb{C}}$ (uniquely determined) are such
that $Z= z+\bar{z}$ and $W= w+\bar{w}$. Since $\alpha \in
E^{*}\otimes \bar{E}^{*}$ is skew-Hermitian, $B$ is skew-symmetric
and its complexification satisfies (\ref{cerinta}) (easy check).
This concludes our claim.

\end{proof}

\subsection{Analogy with skew-symmetric generalized complex
structures}\label{analogies}

The theory of symmetric generalized complex structures from the
previous section is  similar to the theory of skew-symmetric
generalized complex structures developed by Gualtieri in
\cite{gualtieri} and owing to this, one can treat these two types
of structures in a unified way. It is well-known (see e.g.
\cite{gualtieri}) that complex and symplectic structures define
skew-symmetric generalized complex structures and this corresponds
to Example \ref{example} i) and ii) from the previous section. In
the same framework, Theorem \ref{b-field} above is analogous to
Theorem 4.13 from \cite{gualtieri}, which states that any
skew-symmetric generalized complex structure, is, modulo a
$B$-field transformation, the direct sum of a skew-symmetric
generalized complex structure of symplectic type and of  one of
complex type.

The following unified description of the holomorphic space of
symmetric and skew-symmetric generalized complex structures on
vector spaces is a rewriting of Proposition \ref{holo} from the
previous section and of Propositions 2.6 and 4.4 from
\cite{gualtieri}. We shall use it in the statement of Theorem
\ref{main-thm}.

\begin{cor}\label{common} A complex subspace $L\subset (V\oplus
V^{*})^{\mathbb{C}}$ is the holomorphic space of a symmetric or,
respectively, skew-symmetric generalized complex structure 
if and only if it is of the form
\begin{equation}\label{tau-1}
L= L^{\tau}(E, \alpha ) = \{ X+\xi \in E\oplus
(V^{\mathbb{C}})^{*}, \quad \xi\vert_{\tau (E)} = i_{X}\alpha \}
\end{equation}
where $E \subset V^{\mathbb{C}}$ is a complex subspace with
$E+\bar{E}=V^{\mathbb{C}}$ and $\alpha\in E^{*}\otimes
\tau({E})^{*}$ is complex bilinear, such that
\begin{equation}\label{tau-2}
\alpha (X, \tau(Y)) + \tau (\alpha  (Y, \tau (X) )) = 0, \quad
\forall X,Y\in E
\end{equation}
and $\mathrm{Im} (\alpha \vert_{\Delta})$ is non-degenerate (where
$\Delta \subset V$,  $\Delta^{\mathbb{C}} = E \cap \bar{E}$).

In (\ref{tau-1}) and (\ref{tau-2}) the maps $\tau : V^{\mathbb{C}}
\rightarrow V^{\mathbb{C}}$ and $\tau : \mathbb{C}\rightarrow
\mathbb{C}$ are both the standard conjugations, respectively both
the identity maps.
\end{cor}

\subsection{Remarks on integrability}\label{integr}

The generalized tangent bundle $\mathbb{T}M = TM\oplus T^{*}M$ of
a smooth manifold $M$ has a canonical metric of neutral signature,
defined like in (\ref{g-can}), and the theory developed in the
previous sections extends pointwise to manifolds, in an obvious
way.

\begin{defn} A generalized complex structure on a manifold $M$ is
a smooth field of endomorphisms $\mathcal J$ of $\mathbb{T}M$,
which, at any $p\in M$, is a generalized complex structure on
${T}_{p}M.$
\end{defn}

\begin{rem}{\rm
As opposed to the usual terminology, we do not assume that
generalized complex structures  on manifolds are Courant integrable (see
below). In fact, the generalized complex structures we are mainly
interested in, namely,  the symmetric ones, turn out not to be
Courant integrable (see Lemma \ref{non-integrable}).}
\end{rem}

\begin{defn}
A generalized complex structure $\mathcal J$  on a manifold $M$ is called
Courant integrable if the space of sections of its holomorphic
bundle (the $i$-eigenbundle of $\mathcal J$) 
 is closed under the Courant bracket, defined by
$$
[X+\xi , Y+\eta ] = [X, Y] + L_{X}\eta  - L_{Y}\xi +\frac{1}{2}
\left( \xi (Y) - \eta (X))\right) ,
$$
for any vector fields $X$, $Y$ and $1$-forms  $ \xi , \eta $.

\end{defn}

The holomorphic bundle $L\subset \mathbb{T}^{\mathbb{C}}M$ of a
symmetric or skew-symmetric generalized  complex structure on $M$
may be described in terms of a complex subbundle $E\subset
T^{\mathbb{C}}M$ (the image of $L$ through the natural projection
$\mathbb{T}^{\mathbb{C}}M \rightarrow T^{\mathbb{C}}M$) and a
smooth section $\alpha \in \Gamma (E^{*}\otimes \tau (E)^{*})$,
with the algebraic properties from Corollary \ref{common} (we
assume that all points are regular, i.e. $E$ is a genuine complex
vector bundle). There is a basic result of Gualtieri (see
Proposition 4.19 of \cite{gualtieri}) which expresses the Courant
integrability of a skew-symmetric generalized complex structure in
terms of its holomorphic bundle. Since we shall use it repeatedly,
we state it here:

\begin{prop}\cite{gualtieri}\label{gualtieri-prop} A skew-symmetric generalized complex structure
on a manifold $M$, with holomorphic bundle $L = L(E, \alpha )$, is
Courant integrable, if and only if the subbundle $E\subset
T^{\mathbb{C}}M$ is involutive and $d_{E}\alpha =0$, where
$d_{E}\alpha\in\Gamma ( \Lambda^{3}E^{*})$ is the exterior
differential of $\alpha\in \Gamma (\Lambda^{2}E^{*})$, defined by
\begin{align*}
(d_{E}\alpha )(X, Y, Z) &:= X\left(\alpha (Y, Z)\right) + Z\left(
\alpha (X, Y)\right) +Y\left( \alpha (Z,X)\right)
\\
& + \alpha ( X, [Y,Z]) +\alpha (Z, [X, Y]) +\alpha (Y, [Z, X]),
\end{align*}
for any $X, Y, Z\in \Gamma (E).$ \end{prop}

The following simple lemma holds.

\begin{lem}\label{non-integrable}
A symmetric generalized complex structure is never Courant
integrable.\end{lem}

\begin{proof} As proved in  Proposition 3.26 of \cite{gualtieri},
a Courant integrable subbundle of $\mathbb{T}^{\mathbb{C}}M$ is
either $g_{\mathrm{can}}$-isotropic  or of the form $(\Delta
\oplus T^{*}M)^{\mathbb{C}}$, where $\Delta \subset TM$ is
involutive (and non-trivial). Hence, it cannot be the holomorphic
bundle $L$ of a symmetric generalized complex structure (recall
that $L$ is $g_{\mathrm{can}}$-orthogonal to $\bar{L}$, $L\oplus
\bar{L}= \mathbb{T}^{\mathbb{C}}M$ and $g_{\mathrm{can}}$ is
non-degenerate).

\end{proof}

\section{Integrable complex structures on cotangent manifolds}\label{main-section}

Let $(M, {\mathcal J} , D)$ be a manifold with a generalized
complex structure $\mathcal J$ and  a linear connection $D$. The
connection $D$ acts on the cotangent bundle $\pi : T^{*}M
\rightarrow M$ and induces a decomposition
\begin{equation}\label{ide}
T(T^{*}M) = H^{D} \oplus T^{\mathrm{vert}} (T^{*}M) = \pi^{*}
(\mathbb{T}M)
\end{equation}
into horizontal and vertical subbundles. Above, we identified the
horizontal bundle $H^{D}$ with $\pi^{*}(TM)$ and the vertical
bundle $T^{\mathrm{vert}}(T^{*}M)$ of the projection $\pi$ with
$\pi^{*}(T^{*}M).$ From now on, we shall use systematically,
without mentioning explicitly,  the identification (\ref{ide})
between $T(T^{*}M)$ and $\pi^{*}(\mathbb{T}M)$.

\begin{defn}\label{a-c} The almost complex structure $J^{\mathcal J , D}:= \pi^{*}({\mathcal J})$
on the cotangent manifold $T^{*}M$  is called the almost complex
structure defined by $\mathcal J$ and $D$.
\end{defn}

In this section we study the integrability of $J^{{\mathcal J},
D}$, under the assumption that $\mathcal J$ is symmetric or
skew-symmetric. We begin by fixing notation.

\begin{notations}{\rm In computations, we shall use the
notation $\tilde{X}\in {\mathcal X}(T^{*}M)$ for the
$D$-horizontal lift of a vector field $X\in {\mathcal X}(M).$
Forms of degree one on $M$ will be considered as constant vertical
vector fields on the cotangent manifold $T^{*}M.$ With these
conventions, the various Lie brackets $[\cdot , \cdot ]_{\mathcal
L}$  of vector fields on $T^{*}M$ are computed as follows:
\begin{equation}\label{lie}
[\tilde{X},\tilde{Y}]_{\mathcal L} (\gamma ) =
{[X,Y]}^{\widetilde{}}(\gamma ) + R^{D}_{X, Y}(\gamma ),\quad
[\tilde{X}, \xi ]_{\mathcal L} = D_{X}\xi , \quad [\xi , \eta
]_{\mathcal L} =0
\end{equation}
for any $X, Y\in {\mathcal X}(M)$, $\xi , \eta \in \Omega^{1}(M)$
and $\gamma \in T^{*}M$, where
$$
R^{D}_{X,Y} := - D_{X}D_{Y} + D_{Y}D_{X} + D_{[X,Y]}
$$
is the curvature of $D$.}
\end{notations}

The main result from this section is the following.

\begin{thm}\label{main-thm} Let ($M, \mathcal J , D)$ be a
manifold with a  symmetric or skew-symmetric generalized  complex
structure $\mathcal J$ and a linear connection $D$. Let
$L^{\tau}(E, \alpha )$ be the holomorphic bundle of $\mathcal J$,
where $E\subset T^{\mathbb{C}}M$ and $\alpha \in \Gamma
(E^{*}\otimes \tau (E)^{*})$ satisfy the algebraic properties from
Corollary \ref{common}. The almost complex structure $J^{\mathcal J, D}$ from Definition
\ref{a-c} is integrable, if and only if the following conditions
hold:

i) $E$ is an involutive subbundle of $T^{\mathbb{C}}M$;

ii) the complex linear extensions of $D$ and $R^{D}$ satisfy
\begin{equation}\label{adaugat-D}
D_{\Gamma (E)} \Gamma (\tau (E)) \subset \Gamma (\tau (E)), \quad
R^{D}\vert_{E\times E}( \tau (E))=0.
\end{equation}

iii) for any $X, Y, Z\in \Gamma (E)$,
\begin{equation}\label{ec}
(D_{X} \alpha )(Y, \tau (Z)) - (D_{Y} \alpha ) (X, \tau (Z)) +
\alpha (T^{D}_{X}Y, \tau (Z)) =0,
\end{equation}
where $T^{D}$ is the torsion of the connection $D$.

\end{thm}

\begin{proof}
We need to prove that the holomorphic bundle  $\pi^{*}L^{\tau}(E,
\alpha )\subset T^{\mathbb{C}} (T^{*}M)$ of $J^{\mathcal J , D}$
is involutive if and only if the conditions {\it i)}, {\it ii)}
and {\it iii)} from the statement of the theorem hold. For this,
we will compute the Lie brackets of basic sections of
$\pi^{*}L^{\tau}(E, \alpha )$. (By a basic section of 
$\pi^{*}L^{\tau}(E, \alpha )$ we mean  a vector field on the
cotangent manifold $T^{*}M$, of the form $\tilde{X} + \xi$, where
$X +\xi$ is a section of $L^{\tau} (E, \alpha )$). Therefore, let
$X+\xi , Y+\eta\in \Gamma (L^{\tau}(E, \alpha ))$. Then
\begin{equation}\label{epsilon-eta}
X, Y\in \Gamma (E), \quad \xi , \eta \in \Gamma
(T^{\mathbb{C}}M)^{*}, \quad \xi\vert_{\tau (E)} = i_{X}\alpha ,
\quad \eta\vert_{\tau (E)} = i_{Y}\alpha .
\end{equation}
From (\ref{lie}), at any $\gamma \in T^{*}M$,
\begin{equation}\label{Lie}
[\tilde{X}+\xi , \tilde{Y}+\eta ]_{\mathcal L}(\gamma ) =
{[X,Y]}^{\widetilde{}} (\gamma ) + R^{D}_{X,Y}(\gamma ) +
D_{X}\eta - D_{Y}\xi .
\end{equation}
We obtain that $[\tilde{X}+\xi , \tilde{Y}+\eta ]_{\mathcal L}$ is a section of
$\pi^{*}L^{\tau}(E, \alpha )$ if and only if
$$
[X,Y] + R^{D}_{X,Y}(\gamma ) + D_{X}\eta - D_{Y}\xi
$$
belongs to the fiber of $L^{\tau}(E, \alpha )$ at $\pi (\gamma )$,
for any $\gamma \in T^{*}M$, i.e.
\begin{equation}\label{D}
[X, Y]\in \Gamma (E), \quad R^{D}_{X,Y}(\gamma )\vert_{\tau (E)
}=0
\end{equation}
and
\begin{equation}\label{ajutatoare}
(D_{X}\eta  - D_{Y}\xi )(\tau (Z)) =\alpha ( [X,Y], \tau (Z)),
\quad \forall Z\in \Gamma (E).
\end{equation}
We now rewrite (\ref{ajutatoare}). From (\ref{epsilon-eta}), the
left hand side of (\ref{ajutatoare}) is equal to
$$
X\alpha (Y, \tau (Z)) - Y\alpha (X, \tau (Z)) - \eta (D_{X}(\tau
(Z))) +\xi (D_{Y}(\tau (Z)))
$$
and (\ref{ajutatoare}) becomes
\begin{equation}\label{ec-fund}
X\alpha (Y, \tau (Z)) - Y\alpha (X, \tau (Z)) - \eta (D_{X}(\tau
(Z))) +\xi (D_{Y}(\tau (Z))) =\alpha ( [X,Y], \tau (Z)),
\end{equation}
for any $Z\in \Gamma (E).$ From (\ref{epsilon-eta}) again,
$\xi\vert_{\tau (E)}= i_{X}\alpha$, but $\xi$ can take any values
on a complement of $\tau (E)$ in $T^{\mathbb{C}}M.$ Similarly, the
only obstruction on $\eta$ is $\eta\vert_{\tau (E)} = i_{Y}
\alpha$. Thus, if (\ref{ec-fund}) holds for any sections $X+\xi$
and $Y+\eta$ of $L^{\tau}(E, \alpha )$, then
$$
D_{X}(\tau (Z)) \in \Gamma (\tau (E)), \quad \forall X, Z\in
\Gamma (E)
$$
and relation (\ref{ec-fund}) becomes (\ref{ec}). We proved that
$\pi^{*}L^{\tau}(E, \alpha )$ is involutive if and only if
\begin{equation}\label{cond}
[ \Gamma (E), \Gamma (E)]\subset \Gamma (E), \quad
R^{D}\vert_{E\times E}\tau (E)=0, \quad D_{\Gamma (E)} \Gamma
(\tau (E)) \subset \Gamma (\tau (E))
\end{equation}
and relation (\ref{ec}) holds. Our claim follows.
\end{proof}

We end this section by relating  the Courant integrability of a
skew-symmetric generalized complex structure $\mathcal J$ with the
integrability of the almost complex structure $J^{\mathcal J ,
D}.$ This is a straightforward application of Theorem
\ref{main-thm}.

\begin{cor}\label{remarca} Let
$\mathcal J$ be a skew-symmetric generalized complex structure,
with holomorphic bundle $L(E, \alpha )$, and $D$ a linear
connection on $M$. Suppose that $E$ is involutive, $ D_{\Gamma (E)} \Gamma (E) \subset
\Gamma (E)$, $R^{D}_{E,E}E=0$ and the relation 
\begin{equation}\label{rel-2}
(D_{Z}\alpha ) (X,Y) + \alpha (T^{D}_{Z}X,Y) +\alpha
(X,T^{D}_{Z}Y)=0, \quad \forall X,Y, Z\in \Gamma (E)
\end{equation}
holds. Then $J^{\mathcal J , D}$ is integrable if and only if
$\mathcal J$ is Courant integrable.
\end{cor}

\begin{proof}
From Proposition \ref{gualtieri-prop} and Theorem \ref{main-thm}, we need to prove 
that $d_{E}\alpha =0$ is equivalent to  (\ref{ec}) (with $\tau :TM \rightarrow TM$ the
identity map). This is a consequence of (\ref{rel-2}) and the following general identity:
for any $2$-form $\beta$ and vector fields $X, Y, Z$, 
\begin{align}
\nonumber&(D_{X}\beta )(Y, Z) - (D_{Y}\beta )(X, Z) + \beta (
T^{D}_{X}Y,
Z)\\
\label{simple-ec}& = (d\beta )(X, Y, Z) - \left( (D_{Z}\beta )(X,
Y) + \beta (T^{D}_{Z}X, Y) + \beta (X, T^{D}_{Z}Y)\right) .
\end{align}

\end{proof}

\begin{exa}\label{Courant-Lie} {\rm Let $\mathcal J$ be a left-invariant skew-symmetric generalized
complex structure on a Lie group $G$ and $D^{c}$ the (flat)
left-invariant connection on $G$ given by $D^{c}_{X}Y= [X, Y]$,
for any left-invariant vector fields $X, Y.$ Then $D^{c}$
satisfies (\ref{rel-2}), for any left-invariant $2$-form $\alpha$.
We obtain that $\mathcal J$ is Courant integrable if and only if
$J^{\mathcal J , D^{c}}$ is integrable.}
\end{exa}

\section{Complex structures on cotangent manifolds of Lie
groups}\label{complex}

We begin by recalling basic facts we need about semisimple  Lie algebras.

\subsection{Semisimple Lie algebras}\label{semi}

Let $\gg^{\mathbb{C}}$ be a complex semisimple Lie algebra and
\begin{equation}\label{cartan}
\gg^{\mathbb{C}} = \gh + \gg (R) = \gh + \sum_{\alpha\in R}
\gg_{\alpha}
\end{equation}
a Cartan decomposition. We identify
$\gh$ with $\gh^{*}$, using the restriction of the Killing form $B$ of
$\gg^{\mathbb{C}}$ to $\gh$. By
means of this identification, we denote by
$\gh_{\mathbb{R}}\subset \gh$ the real span of the set of roots
$R\subset \gh^{*}$ of $\gg^{\mathbb{C}}$ relative to $\gh$ and by
$H_{\alpha}\in \gh_{\mathbb{R}}$ the vector which corresponds to
the root $\alpha \in R.$ Recall that a Weyl basis of the root part
$\gg (R):=\sum_{\alpha\in R} \gg_{\alpha}$  consists of root
vectors $\{E_{\alpha} , \alpha\in R\}$, satisfying the following
conditions:
$$
[ E_{\alpha}, E_{-\alpha }] = H_{\alpha},\quad B(E_{\alpha},
E_{-\alpha }) = 1,\quad N_{-\alpha , -\beta } = - N_{\alpha\beta},
\quad N_{\alpha\beta}\in \mathbb{R},
$$
where the structure constants $N_{\alpha\beta}$ are defined by
$$
[E_{\alpha}, E_{\beta }] = N_{\alpha\beta} E_{\alpha +\beta},
\quad \forall\alpha , \beta , \alpha  + \beta \in R.
$$
A simple argument which uses the Jacobi identity for $E_{\alpha}$,
$E_{\beta}$, $E_{\gamma}$ shows that for any $\alpha , \beta ,
\gamma \in R$, such that $\alpha + \beta + \gamma =0$,
\begin{equation}\label{rel-helg}
N_{\alpha\beta} = N_{\beta\gamma} = N_{\gamma\alpha}
\end{equation}
(see e.g. \cite{helg}, page 146).

Recall now that a real form of $\gg^{\mathbb{C}}$ is the fixed
point set of an antilinear involution
$$
\sigma : \gg^{\mathbb{C}} \rightarrow\gg^{\mathbb{C}}, \quad
x\rightarrow \sigma (x) = \bar{x},
$$
i.e. an automorphism of real Lie algebras, which is complex
antilinear and satisfies $\sigma^{2}= \mathrm{Id}$. We review,
following Theorem 6.88 of \cite{knapp}, the structure of such real
forms. The idea is that  $\gg$  is
determined (up to isomorphism) by its Vogan diagram, which is the
Dynkin diagram of $\gg^{\mathbb{C}}$ (representing a set of simple
roots $\Pi$ relative to a chosen Cartan subalgebra $\gh$) together with
two pieces of data: an involutive automorphism $\theta : \Pi
\rightarrow \Pi$ of the Dynkin diagram and some painted nodes, in
the fixed point set of $\theta .$ 
Chose a Weyl basis
$\{ E_{\alpha}\}$ of $\gg (R)$, where
$R= [\Pi ]$ is the set of roots of $\gg^{\mathbb{C}}$ relative to $\gh$.  The action of $\theta$ on $\Pi$ extends by linearity to
$\gh_{\mathbb{R}}^{*}\cong \gh_{\mathbb{R}}$ and this action
preserves $R$. The
antiinvolution $\sigma$ preserves $\gh$ and it acts on $R$ by
$$
\sigma : R \rightarrow R, \quad \sigma (\alpha ):=
\overline{\alpha \circ \sigma }.
$$
This action coincides, up to a minus sign, with the action of
$\theta$: $\sigma\vert_{R} = -\theta\vert_{R}$. On root vectors from the chosen Weyl basis, $\sigma$
acts as
\begin{equation}\label{action-sigma}
\sigma (E_{\alpha}) = -a_{\alpha}E_{\sigma (\alpha )},\quad \alpha
\in R,
\end{equation}
where $\{ a_{\alpha}, \alpha \in R\}$ (determined by the painted
nodes in the Vogan diagram) is a set of constants, satisfying
\begin{equation}\label{sim-1}
a_{\alpha} = a_{-\alpha} =a_{\sigma ( \alpha )}\in \{ \pm 1\},
\quad \forall \alpha \in R
\end{equation}
and
\begin{equation}\label{sim-2}
a_{\alpha +\beta} = - a_{\alpha}a_{\beta}N_{\alpha\beta}^{-1}
N_{\sigma (\alpha ) \sigma (\beta )},\quad \alpha , \beta ,\alpha
+\beta \in R.
\end{equation}
The real form $\gh_{\gg} = \gh^{\sigma}= \gh^{+} +\gh^{-}$,
where
\begin{equation}\label{cartan-real}
\gh^{+}:= \langle i(H_{\alpha} + H_{- \sigma (\alpha )}),\ \alpha
\in R\rangle , \quad \gh^{-}:= \langle H_{\alpha} + H_{\sigma
(\alpha )},\ \alpha \in R\rangle
\end{equation}
(the sign $\langle \cdots \rangle$ means the real span of the
respective vectors) is a Cartan subalgebra of $\gg$. 
Up to isomorphism, $\gg$ can be recovered from its Vogan diagram 
as
\begin{equation}\label{alg}
 \gg = (\gg^{\mathbb{C}})^{\sigma}  
= \gh_{\gg} +\sum_{\alpha\in R}  \mathbb{R} (E_{\alpha} - a_{\alpha}
 E_{\sigma (\alpha )})+\sum_{\alpha \in R} \mathbb{R} i (E_{\alpha} +
 a_{\alpha}E_{\sigma ( \alpha )}).
\end{equation}

\begin{rem}{\rm Since $\theta$ permutes $\Pi$, there is no root $\alpha\in R$ such
that $\sigma (\alpha ) =\alpha$. This means that $\gh_{\gg}$ is a
maximally compact Cartan subalgebra of $\gg$ (see Proposition 6.70
of \cite{knapp}). The real form  $\gg$ (and any Lie group $G$ with
Lie algebra $\gg$) is called of inner type if $\sigma ( \alpha )=
-\alpha$ for any $\alpha \in R$ (the automorphism $\theta$ of the
Vogan diagram is the identity). Any compact real form is of inner
type, with $a_{\alpha}=1$, for any $\alpha \in R.$ A real form
$\gg$ (and any Lie group $G$ with Lie algebra $\gg$) which is not
of inner type is called of outer type. For more details on real
semisimple Lie algebras, Vogan diagrams, maximally compact Cartan
subalgebras, see e.g. \cite{knapp}, Chapter VI.}
\end{rem}

\subsection{Admissible triples on Lie
groups}\label{semi1}

Let $G$ be a Lie group. We identify $T_{e}G$ with the space of
left-invariant vector fields on $G$ and with the Lie algebra $\gg$
of $G$,  in the usual way. The following definition encodes the
conditions from Theorem \ref{main-thm}, when $M=G$ and $\mathcal
J$, $D$ are left-invariant. Recall that a connection $D$ is left-invariant if
$D_{X}Y$ is left-invariant, when $X$ and $Y$ are so.

\begin{defn}\label{adm} A (symmetric, respectively skew-symmetric)
$\gg$-admissible triple is a triple $(\gk , \mathcal D , \epsilon  )$,
with the following
properties:

i) $\gk$ is a complex subalgebra of $\gg^{\mathbb{C}}$,  such that
$\gk + \bar{\gk} = \gg^{\mathbb{C}}$;

ii) $\mathcal D : \gg \times \gg \rightarrow \gg$, $(X, Y)
\rightarrow {\mathcal D}_{X}(Y)$,  is a bilinear map whose complex
linear extension satisfies
\begin{equation}\label{invers}
{\mathcal D}_{\gk} \tau ({\gk}) \subset\tau ({\gk})
\end{equation}
and
\begin{equation}\label{curv-alg}
R^{\mathcal D}_{X,Y}Z:= - \mathcal D_{X}\mathcal D_{Y}(Z) +
\mathcal D_{Y}\mathcal D_{X} (Z) + \mathcal D_{[X,Y]} (Z)=0, \quad
\forall X, Y\in \gk ,\quad \forall Z\in \tau (\gk ).
\end{equation}
iii) $\epsilon \in \gk^{*} \otimes\tau(\gk )^{*}$  satisfies
$$
\epsilon (X, \tau (Y)) + \tau ( \epsilon (Y, \tau (X)))=0,\quad \forall X, Y\in \gk
$$
and
\begin{equation}\label{epsilon-eqn}
\epsilon (X, \mathcal D_{Y}(\tau (Z))) - \epsilon (Y, \mathcal
D_{X} (\tau (Z ))) = \epsilon ( [X, Y], \tau (Z)), \quad \forall X,
Y, Z\in \gk .
\end{equation}
Moreover, $g_{\Delta}:= \mathrm{Im}(\epsilon\vert_{\Delta})$ is 
non-degenerate on $\Delta = (\gk \cap \bar{\gk})^{\sigma}$.

Above, the maps $\tau : \gg^{\mathbb{C}}\rightarrow \gg^{\mathbb{C}}$
and $\tau : \mathbb{C}\rightarrow
\mathbb{C}$ are both the standard conjugations, respectively both
the identity maps.

\end{defn}

The following correspondence holds and will play a key role in our
treatment from the next subsection.

\begin{prop}\label{cons}
There is a one to one correspondence between:

i) pairs $(\mathcal J , D)$ formed by a left-invariant (symmetric,
respectively skew-symmetric) generalized  complex structure
$\mathcal J$ and a left-invariant connection $D$ on 
$G$, such that the associated almost complex structure
$J^{\mathcal J , D}$ on
$T^{*}G$ is integrable;

ii)  (symmetric, respectively skew-symmetric)
$\gg$-admissible triples $(\gk ,
\mathcal D , \epsilon )$.

In this correspondence $\mathcal D$ is the restriction of $D$ to
the space of left-invariant vector fields, $\gk : = E_{e}$ and
$\epsilon := \alpha\vert_{\gk\times \tau(\gk )}$, where
$L^{\tau}(E, \alpha )$ is the holomorphic bundle of $\mathcal J .$

\end{prop}

\begin{proof}
Using the left-invariance of $E$ and 
$\alpha$,  one may check
that the conditions from Theorem \ref{main-thm}, on the
integrability of $J^{\mathcal J , D}$, become the conditions for
$(\gk , \mathcal D , \epsilon )$ to be a $\gg$-admissible triple.
For example, to prove the equivalence between (\ref{ec}) and
(\ref{epsilon-eqn}), we notice that (\ref{ec}) holds if and only
if it holds for any $X, Y, Z\in \Gamma (E)$ left-invariant, and
for such arguments,  $\alpha (Y, \tau (Z))$ and $\alpha (X, \tau (Z))$ are constant  (because $\alpha$
is left-invariant).
\end{proof}

\subsection{Regular admissible triples and regular generalized
complex structures}\label{semi2}

Here and until the end of Section \ref{complex} we 
fix a complex semisimple Lie algebra $\gg^{\mathbb{C}}$, a real form $\gg = (\gg^{\mathbb{C}})^{\sigma}$ 
given by (\ref{alg}), and a Lie group $G$ with Lie algebra $\gg .$ 
A (complex)
subalgebra $\gk \subset \gg^{\mathbb{C}}$ is called regular, if it is
normalized by the (maximally compact) Cartan subalgebra
$\gh_{\gg}$ of $\gg .$ It is known
(see e.g. \cite{vinberg}, Proposition 1.1, page 183) that such a subalgebra is
of the form
\begin{equation}\label{g-k}
\gk = \gh_{\gk} + \gg (R_{0}) = \gh_{\gk} + \sum_{\alpha \in
R_{0}}\gg_{\alpha}
\end{equation}
where $\gh_{\gk} = \gk \cap \gh$  and $R_{0}\subset R$ is a
closed subset of roots (i.e. if $\alpha , \beta \in R_{0}$ and
$\alpha +\beta\in R$, then $\alpha +\beta\in R_{0})$. Remark that
\begin{equation}\label{sf}
\bar{\gk} = \sigma (\gk ) =\bar{\gh}_{\gk} +  \sum_{\alpha \in R_{0}}\gg_{\sigma
(\alpha )}, \quad \gk \cap \bar{\gk} =  \gh_{\gk} \cap \bar{\gh}_{\gk} + 
 \sum_{\alpha \in R_{0}\cap \sigma (R_{0})}\gg_{\alpha }.
\end{equation}

\begin{defn}\label{regular-def} Let $\mathcal J$ be  a left-invariant
(symmetric or skew-symmetric) generalized complex structure
on $G$ and $L^{\tau }(\gk , \epsilon )$ the fiber at $e\in G$ of its
holomorphic bundle. Then $\mathcal J$ is called regular if $\gk$ is
a regular subalgebra of $\gg^{\mathbb{C}}$. Similarly, a
$\gg$-admissible triple $(\gk , \mathcal D , \epsilon )$ is called
regular if $\gk$ is a regular subalgebra of $\gg^{\mathbb{C}}.$
\end{defn}

We need to recall the notions of $\sigma$-parabolic and $\sigma$-positive systems \cite{liana}. 
They reduce, when $\gg$ is of inner type, to the usual notions of
parabolic and positive root systems, respectively.

\begin{defn}\label{def-sigma} A closed set of roots $R_{0}\subset R$ 
is called a $\sigma$-parabolic system, if $R_{0}\cup \sigma (R_{0}) =R.$
If, moreover,  $R_{0} \cap \sigma (R_{0}) = \emptyset$, then
$R_{0}$ is called a $\sigma$-positive system.\end{defn}

The following simple lemma holds.

\begin{lem}\label{simple-aj} If a regular subalgebra $\gk$ as in (\ref{g-k})
belongs to a $\gg$-admissible triple, then its root part $R_{0}$
is a $\sigma$-parabolic system and its Cartan part $\gh_{\gk}$
satisfies $\gh_{\gk} +\bar{\gh}_{\gk} = \gh .$ If, moreover,
$R_{0}$ is a $\sigma$-positive system, then $\gk \cap \bar{\gk }=
\gh_{\gk} \cap \bar{\gh}_{\gk}.$

\end{lem}

\begin{proof}
From the definition of $\gg$-admissible triples, $\gk + \bar{\gk}
= \gh$. This relation, together with (\ref{sf}), implies the
statement of the lemma.
\end{proof}

\subsection{Complex structures on $T^{*}G$}\label{complex-main}

Our aim in this section is to define a natural left-invariant
connection $D^{0}$ on $G$ and to determine all regular symmetric
generalized complex structures $\mathcal J$, with the property
that the almost complex structure $J^{\mathcal J, D^{0}}$ on
$T^{*}G$ is integrable, or, equivalently, the associated triple
$(\gk , \mathcal D^{0} , \epsilon )$ is $\gg$-admissible (and
regular, symmetric). From definition, a bilinear map $\mathcal D :
\gg\times \gg \rightarrow \gg$ can belong to a symmetric
$\gg$-admissible triple $(\gk , \mathcal D , \epsilon )$ only if
its complex linear extension $\mathcal D : \gg^{\mathbb{C}}
\times\gg^{\mathbb{C}} \rightarrow \gg^{\mathbb{C}}$ satisfies
(\ref{invers}) (with $\tau = \sigma$, hence $\tau (\gk ) = \bar{\gk}$)
and (\ref{curv-alg}). Recall now that $\gk$ is of the
form (\ref{g-k}) and $\bar{\gk}$ of the form (\ref{sf}). From
these relations, it is immediate that if
\begin{equation}\label{cond-c}
\mathcal D_{\gg_{\alpha}} (\gg_{\beta}) \subset \gg_{\sigma (\alpha )
+\beta}, \quad \mathcal D_{\gh}(\gg_{\beta}) \subset \gg_{\beta}, \quad
\mathcal D_{\gg_{\beta}}(\gh ) \subset \gg_{\sigma (\beta )}, \quad
\mathcal D_{\gh} (\gh )=0,
\end{equation}
for any $\alpha , \beta \in R$ (with $\gg_{\sigma (\alpha ) +\beta}:=0$ if 
$\alpha +\sigma (\beta )\notin R$),  then is  (\ref{invers}) is satisfied. 
A map whose complex linear
extension satisfies (\ref{cond-c}) and (\ref{curv-alg}) is
provided by the following lemma.

\begin{lem}\label{connection} Let $\mathcal D^{0}: \gg^{\mathbb{C}}\times
\gg^{\mathbb{C}}\rightarrow \gg^{\mathbb{C}}$ be a complex
bilinear map given by
\begin{align*}
\mathcal D^{0}_{E_{\alpha}}(E_{\beta} )&=- a_{\alpha} [E_{\sigma
(\alpha ) }, E_{\beta}],
\quad \mathcal D^{0}_{H}(E_{\beta})=\sigma (\beta) (H)E_{\beta},\\
\mathcal D^{0}_{E_{\beta}}(H )&= \sigma (\beta) (H)a_{\beta}
E_{\sigma (\beta )}, \quad \mathcal D^{0}_{H}(\tilde{H})=0
\end{align*}
for any $\alpha , \beta \in R$ and $H, \tilde{H}\in \gh$. Then
$\mathcal D^{0}$ satisfies
\begin{equation}\label{properties-req}
\mathcal D^{0}_{\gk } (\bar{\gk}) \subset \bar{\gk},\ \mathcal
D^{0}_{\gg} ( \gg ) \subset\gg,\ R^{\mathcal D^{0}}=0.
\end{equation}
\end{lem}

\begin{proof}
We already explained that $\mathcal D^{0}_{\gk } (\bar{\gk})
\subset \bar{\gk}$. We now prove $\mathcal D^{0}_{\gg} ( \gg )
\subset\gg .$ For any $\alpha \in R$, let
$A_{\alpha}:= E_{\alpha} -a_{\alpha} E_{\sigma ( \alpha)}$ and
$B_{\alpha}:= i (E_{\alpha}+ a_{\alpha}E_{\sigma (\alpha)})$. By  a
straightforward computation, which uses (\ref{sim-1}), we obtain:
\begin{align*}
\mathcal D^{0}_{A_{\alpha}}(A_{\beta})&=- a_{\alpha} \left([
E_{\sigma (\alpha )}, E_{\beta}] + a_{\alpha}a_{\beta}
[E_{\alpha}, E_{\sigma (\beta )}] \right) + \left( [E_{\alpha},
E_{\beta}] +a_{\alpha}a_{\beta}[E_{\sigma (\alpha )}, E_{\sigma
(\beta )}]\right) \\
\mathcal D^{0}_{B_{\alpha}}(B_{\beta})&=\left( [E_{\alpha}, E_{\beta}]
+a_{\alpha}a_{\beta}[E_{\sigma (\alpha )}, E_{\sigma (\beta )}]\right)
+a_{\alpha}\left( [ E_{\sigma (\alpha )}, E_{\beta}] +
a_{\alpha}a_{\beta} [E_{\alpha}, E_{\sigma (\beta )}]\right)\\
\mathcal D^{0}_{A_{\alpha}}(B_{\beta})&= i \left([E_{\alpha},
E_{\beta}] -a_{\alpha}a_{\beta}[E_{\sigma (\alpha )}, E_{\sigma
(\beta )}]\right) +a_{\beta}i\left( [E_{\alpha}, E_{\sigma (\beta
)}] -a_{\alpha}a_{\beta }[E_{\sigma (\alpha )}, E_{\beta
}]\right)\\
\mathcal D^{0}_{B_{\alpha}}(A_{\beta})&= -i \left( [E_{\alpha},
E_{\beta}] -a_{\alpha}a_{\beta}[E_{\sigma (\alpha )}, E_{\sigma
(\beta )}]\right) - ia_{\alpha}\left( [E_{\sigma (\alpha )},
E_{\beta }]- a_{\alpha}a_{\beta}[E_{\alpha}, E_{\sigma (\beta
)}]\right).
\end{align*}
Moreover, for any $\alpha ,
\beta \in R$ and $H\in \gh^{+}$,
\begin{align*}
\mathcal D^{0}_{A_{\alpha}}(H)= i\alpha (H)a_{\alpha}B_{\sigma
(\alpha) }, \quad
\mathcal D^{0}_{B_{\alpha}}(H) = i\alpha (H) A_{\alpha}\\
\mathcal D^{0}_{H}(A_{\alpha}) =i\alpha (H)B_{\alpha},\quad
\mathcal D^{0}_{H}(B_{\alpha})= -i\alpha (H)A_{\alpha},
\end{align*}
while for any $\alpha , \beta \in R$ and  $H\in \gh^{-}$,
\begin{align*}
\mathcal D^{0}_{A_{\alpha}}(H)= \alpha (H)a_{\alpha}A_{\sigma
(\alpha) }, \quad
\mathcal D^{0}_{B_{\alpha}}(H) = \alpha (H) B_{\alpha}\\
\mathcal D^{0}_{H}(A_{\alpha}) =\alpha (H)A_{\alpha},\quad
\mathcal D^{0}_{H}(B_{\alpha})= \alpha (H)B_{\alpha}.\\
\end{align*}
For any $\alpha , \beta \in R$, the expressions 
$$[E_{\alpha} , E_{\beta}] +a_{\alpha}a_{\beta}[E_{\sigma (\alpha )},
E_{\sigma (\beta )}], \quad i \left( [E_{\alpha} , E_{\beta}] -a_{\alpha}a_{\beta}[E_{\sigma (\alpha )},
E_{\sigma (\beta )}]\right) 
$$ 
belong to $\gg$ and $a_{\alpha} = a_{\sigma (\alpha )}\in \{\pm 1\}$.   Moreover, any root takes real values on
$\gh^{-}$ and purely imaginary values on $\gh^{+}$. Therefore,  the above
computations show that $\mathcal D^{0}_{\gg } (\gg )\subset \gg$,
as required.  It is straightforward to check, using the definition
of ${\mathcal D^{0}}$, that $R^{\mathcal D^{0}} =0$.  Our claim follows.
\end{proof}

The preferred connection we were looking for is defined as follows.

\begin{defn}\label{def-d-0} The connection $D^{0}$ is the unique (flat) left-invariant connection on $G$ which on left-invariant vector fields 
coincides with the map $\mathcal D^{0}$ from Lemma \ref{connection}.
\end{defn}

With the above preliminary considerations, we can now state our main result from this section. Below 
we denote by $\{ \omega_{\alpha}\in
(\gg^{\mathbb{C}})^{*}, \alpha \in R \}$ the covectors defined by
$\omega_{\alpha}(E_{\beta}) = \delta_{\alpha\beta}$ for any
$\alpha , \beta\in R$ and $\omega_{\alpha}\vert_{\gh}=0.$ We 
use the notation $R_{0}^{\mathrm{sym}} := R_{0}\cap  (- R_{0})$
for the symmetric part of $R_{0}.$

\begin{thm}\label{main-applic} Consider a triple $(\gk , \mathcal D^{0},
\epsilon )$, with $\gk$ the regular subalgebra (\ref{g-k}),
$\mathcal D^{0}$  as in Lemma \ref{connection} and $\epsilon
\in \gk^{*}\otimes \bar{\gk}^{*}$ skew-Hermitian. Assume that
\begin{equation}\label{assumption-not}
(\alpha +\beta)\vert_{\gh_{\gk}}\neq 0, \quad  \forall\alpha ,
\beta \in R_{0}\cup \{ 0\}, \quad \alpha +\beta \neq 0.
\end{equation}
Then $(\gk , \mathcal D^{0}, \epsilon )$ is a (symmetric)
$\gg$-admissible triple  (and 
the associated pair $(\mathcal J , D^{0})$ defines a complex structure 
$J^{\mathcal J , D^{0}}$ on $T^{*}G$) if and
only if the following conditions hold:

i)  the root system $R_{0}$ of $\gk$ is a $\sigma$-parabolic
system (see Definition \ref{def-sigma}) and the Cartan part
satisfies $\gh_{\gk } +
\bar{\gh}_{\gk} = \gh$;

ii)  the skew-Hermitian $2$-form $\epsilon \in \gk^{*} \otimes
\bar{\gk}^{*}$ is given by
\begin{align}
\nonumber&\epsilon = \epsilon_{0} + \sum_{\alpha\in R_{0}}
\mu_{\alpha} ( \alpha \otimes \omega_{\sigma(\alpha )}+ a_{\alpha}
\omega_{\alpha}\otimes\sigma (\alpha ) )\\
\nonumber& - \sum_{\alpha , \beta , \alpha +\beta\in
R_{0}}a_{\alpha}\mu_{\alpha +\beta}N_{\sigma(\alpha )\sigma(\beta
)} \omega_{\alpha}\otimes
\omega_{\sigma(\beta )}\\
\label{exp}& + \sum_{\gamma\in R_{0}^{\mathrm{sym}}} \nu_{\gamma}
\omega_{\gamma} \otimes \omega_{-\sigma (\gamma )}
\end{align}
where $\epsilon_{0} \in \gh_{\gk}^{*}\otimes \bar{\gh}_{\gk}^{*}$
is skew-Hermitian (trivially extended to $\gk$), $\mu_{\alpha},
\nu_{\gamma}$ ($\alpha \in R_{0}$, $\gamma \in
R^{\mathrm{sym}}_{0}$) are any real constants, such that the
$\nu_{\alpha}$'s  satisfy
\begin{equation}\label{ad-nu}
\nu_{\alpha} + \nu_{-\alpha } =0, \quad \forall \alpha \in
R_{0}^{\mathrm{sym}}
\end{equation}
and, for any $\alpha , \beta , \gamma\in R^{\mathrm{sym}}_{0}$, with $\alpha + \beta +\gamma =0$,
\begin{equation}\label{suplimentara}
a_{\alpha}\nu_{\alpha} +a_{\beta}\nu_{\beta}
+a_{\gamma}\nu_{\gamma} =0.
\end{equation}

iii) The pseudo-Riemannian metric $g_{\Delta}:=
\mathrm{Im}(\epsilon \vert_{\gk \cap \gg}) $ is non-degenerate and
\begin{equation}\label{e3}
\epsilon_{0} (H, H_{\sigma (\alpha )}) = 0, \quad \forall H\in
\gh_{\gk}, \quad \forall \alpha \in R_{0}^{\mathrm{sym}}.
\end{equation}

\end{thm}

\begin{proof}
From Definition \ref{adm} and Lemma
\ref{simple-aj}, we need to prove that $\epsilon$ satisfies
\begin{equation}\label{epsilon-eqn-1}
\epsilon (X, \mathcal D^{0}_{Y}(\bar{Z})) - \epsilon (Y, \mathcal
D^{0}_{X}(\bar{Z})) = \epsilon ( [X, Y], \bar{Z}), \quad \forall
X, Y, Z\in \gk ,
\end{equation}
with $\mathcal D^{0}$ from Lemma \ref{connection}, if and only if
it is given by (\ref{exp}) and conditions (\ref{ad-nu}),
(\ref{suplimentara}) and (\ref{e3}) are satisfied. In order to
prove this statement, we chose various arguments in
(\ref{epsilon-eqn-1}). Below, $H, \tilde{H}\in \gh_{\gk}$ and
$\alpha , \beta, \gamma\in R_{0}.$ First, let $X:= H$, $Y:=
\tilde{H}$ and $Z:= E_{\alpha}$. With these arguments,
(\ref{epsilon-eqn-1}) becomes
$$
\alpha (\tilde{H}) \epsilon (H, E_{\sigma(\alpha )}) = \alpha (H)
\epsilon (\tilde{H}, E_{\sigma (\alpha )}).
$$
From (\ref{assumption-not}),  $\alpha\vert_{\gh_{\gk}}$ is non-trivial. Chosing $\tilde{H}$ such
that $\alpha (\tilde{H})\neq 0$, we deduce that the above relation is equivalent to
\begin{equation}\label{e1}
\epsilon (H, E_{\sigma (\alpha )}) = \mu_{\alpha} \alpha (H),
\quad \forall H\in \gh_{\gk}, \quad \forall \alpha \in R_{0},
\end{equation}
for a constant $\mu_{\alpha}\in \mathbb{C}.$ By letting $X:= H$,
$Y:= E_{\alpha}$ and $Z:= \tilde{H}$ in (\ref{epsilon-eqn-1}), we
obtain that $\mu_{\alpha}\in \mathbb{R}$, for any $\alpha \in
R_{0}.$

Next, let $X:= E_{\alpha}$, $Y:=H$ and $Z:= E_{\beta}$ in
(\ref{epsilon-eqn-1}). We obtain
$$
\epsilon (E_{\alpha}, \mathcal D^{0}_{H}(\bar{E}_{\beta})) -
\epsilon (H, \mathcal D^{0}_{E_{\alpha}}(\bar{E}_{\beta})) =
\epsilon ([E_{\alpha}, H], \bar{E}_{\beta})
$$
or
\begin{equation}\label{r}
(\alpha + \beta )(H) \epsilon (E_{\alpha}, E_{\sigma (\beta )}) +
a_{\alpha} \epsilon (H, [E_{\sigma (\alpha )}, E_{\sigma (\beta
)}]) =0.
\end{equation}
If  $\alpha +\beta \neq
0$, then $(\alpha +\beta )\vert_{\gh_{\gk}}$ is non-trivial, by  (\ref{assumption-not}),
and
the above relation, together with (\ref{e1}), gives
\begin{align}
\nonumber\epsilon (E_{\alpha}, E_{\sigma (\beta )})& =-
a_{\alpha}\mu_{\alpha +\beta} N_{\sigma (\alpha )\sigma (\beta )},\
\forall\alpha , \beta, \alpha + \beta \in R_{0},\\
\label{e2}\epsilon (E_{\alpha}, E_{\sigma (\beta )}) &=0,\ \forall\alpha , \beta \in R_{0},\ 
\alpha+ \beta\notin R\cup \{ 0\} .
\end{align}
If $\alpha +\beta =0$, relation (\ref{r}) gives (\ref{e3}).

We now remark that conditions (\ref{e1}) and (\ref{e2}) imply that
$\epsilon$ is of the form (\ref{exp}), with $\mu_{\alpha}\in
\mathbb{R}$ ($\alpha \in R_{0}$) and $\nu_{\alpha}
:=\epsilon (E_{\alpha}, E_{-\sigma (\alpha )}) \in \mathbb{C}$
($\alpha \in R_{0}^{\mathrm{sym}}$).

We still need to consider (\ref{epsilon-eqn-1}), with the
remaining two types of arguments:
$X=E_{\alpha}$,
$Y= E_{\beta}$, $Z:= H$, and, respectively,  $X:= E_{\alpha}$, $Y:=
E_{\beta}$, $Z:= E_{\gamma}$  (from the definition of $\mathcal D^{0}$,
(\ref{epsilon-eqn-1}) holds when all $X$, $Y$, $Z$ belong to the
Cartan part $\gh_{\gk}$).

Let $X=E_{\alpha}$, $Y= E_{\beta}$, $Z:= H$. Relation
(\ref{epsilon-eqn-1}) gives
\begin{equation}\label{gres}
\overline{\beta (H)} a_{\beta}\epsilon (E_{\alpha}, E_{\sigma
(\beta )}) + \overline{\alpha (H)} a_{\beta}\overline{\epsilon
(E_{\alpha}, E_{\sigma (\beta )})} = \epsilon ([E_{\alpha},
E_{\beta}], \bar{H}).
\end{equation}
When $\alpha +\beta \neq 0$, relation (\ref{gres}) follows from
(\ref{e1}) and (\ref{e2}) (and the skew-Hermitian property of
$\epsilon$). When $\alpha +\beta = 0$, relation (\ref{gres})
implies that $\nu_{\alpha}\in \mathbb{R}$, for any $\alpha \in
R_{0}^{\mathrm{sym}}.$ Since $\epsilon$ is skew-Hermitian and $\nu_{\alpha}\in \mathbb{R}$,
relation (\ref{ad-nu}) holds.

Finally, let $X:= E_{\alpha}$, $Y:= E_{\beta}$, $Z:= E_{\gamma}$ in
(\ref{epsilon-eqn-1}). From (\ref{e1}), (\ref{e2}) and
$\mu_{\alpha}, \nu_{\beta}\in \mathbb{R}$, relation (\ref{epsilon-eqn-1}) is
automatically satisfied, when $\alpha +\beta +\gamma \neq 0$; when
$\alpha +\beta +\gamma =0$, we obtain
\begin{equation}\label{sup-1}
a_{\beta}N_{\sigma (\beta )\sigma (\gamma)} \nu_{\alpha}
+a_{\alpha}N_{\sigma(\gamma )\sigma (\alpha )} \nu_{\beta}
+N_{\beta\alpha} \nu_{ \gamma} =0.
\end{equation}
Using now the relations
$$
N_{\sigma (\beta )\sigma (\gamma )} = - a_{\beta
+\gamma}a_{\beta}a_{\gamma} N_{\beta\gamma}, \quad N_{\sigma
(\gamma )\sigma (\alpha )} = - a_{\alpha
+\gamma}a_{\alpha}a_{\gamma} N_{\gamma\alpha}
$$
and $N_{\alpha\beta} = N_{\beta\gamma}= N_{\gamma\alpha}$ (because
$\alpha + \beta +\gamma =0$; see
Subsection \ref{semi}), we obtain that (\ref{sup-1}) is equivalent
to (\ref{suplimentara}). Our claim follows.
\end{proof}

The statement of Theorem \ref{main-applic} requires various
comments. First, we need to explain how the
constants $\nu_{\alpha}$ can be constructed, such that (\ref{ad-nu}) and
(\ref{suplimentara}) are satisfied. Next, we need to study the
non-degeneracy of $g_{\Delta}.$ This will be done in the following
paragraphs.

\subsubsection{The construction of $\nu_{\alpha}$}

Let  $R_{0}$ be a $\sigma$-parabolic system of $R$ (the argument holds for
any closed subsystem of $R$, not necessarily $\sigma$-parabolic).
In this paragraph, we describe a method to construct real constants
$\nu_{\alpha}$, $\alpha \in R_{0}^{\mathrm{sym}}$, such that
conditions (\ref{ad-nu}) and (\ref{suplimentara}) from Theorem
\ref{main-thm} hold. Since $R_{0}^{\mathrm{sym}}$ is closed and
symmetric, it is a root system (see e.g. \cite{bourbaki}, page
164). Let $\Pi := \{ \alpha_{0}, \cdots , \alpha_{k}\}$ be a system
of simple roots of ${R}_{0}^{\mathrm{sym}}$. Define, as usual, the
height of $\alpha =n_{1}\alpha_{1}+\cdots + n_{k}\alpha_{k}\in
R_{0}^{\mathrm{sym}}$ with respect to $\Pi$, by $n(\alpha ):=
n_{1}+\cdots + n_{k}.$ 

\begin{lem}\label{nu} The constants $\nu_{\alpha} := a_{\alpha} n(\alpha )$,
for any $\alpha\in R^{\mathrm{sym}}_{0}$, satisfy (\ref{ad-nu})
and (\ref{suplimentara}).
\end{lem}

\begin{proof} The hight function $n: R_{0}^{\mathrm{sym}} \rightarrow \mathbb{Z}$ is
additive. In particular, $n(-\alpha )  = - n(\alpha )$ and if
$\alpha +\beta +\gamma =0$, then $n(\alpha ) + n(\beta ) +n
(\gamma )=0$. Recall also that $a_{\alpha}^{2} = 1$ and
$a_{-\alpha}= a_{\alpha}$ for any $\alpha$. The claim follows.
\end{proof}

\subsubsection{The non-degeneracy of $g_{\Delta}$}

We begin with the simplest case, when $R_{0}$ is a
$\sigma$-positive system.

\begin{rem}\label{outer}{\rm We consider a triple $(\gk , \mathcal D^{0}, \epsilon )$
satisfying the conditions {\it i)} and {\it  ii)} of Theorem
\ref{main-applic}.  We assume, moreover, that $R_{0}$ is a
$\sigma$-positive system (not only $\sigma$-parabolic).
Then $\Delta = (\gk \cap \bar{\gk})^{\sigma}$ reduces
to $(\gh_{\gk} \cap \bar{\gh}_{\gk})^{\sigma}$ and the
non-degeneracy of $g_{\Delta}=\mathrm{Im}(
\epsilon\vert_{\Delta})$  concerns only the Cartan part
$\epsilon_{0}$ of $\epsilon .$ Our aim is to show that, under a
mild additional assumption, we can chose the Cartan part
$\epsilon_{0}$ of $\epsilon $ such that $g_{\Delta}$ is
non-degenerate and (\ref{e3}) is satisfied as well. More
precisely, assume that the subspace
$$
\mathcal S :=\mathrm{Span}_{\mathbb{C}}\{ H_{\alpha},\
\alpha\in R_{0}^{\mathrm{sym}}\}
$$
is transverse to its conjugate
$$
\bar{\mathcal S}=\sigma ({\mathcal S}) =
\mathrm{Span}_{\mathbb{C}}\{ H_{\sigma (\alpha )},\ \alpha\in
R_{0}^{\mathrm{sym}}\}.
$$
(We remark that this holds for many
$\sigma$-positive systems, see  Subsections 5.1-5.3  of
\cite{liana}). A simple argument (see \cite{liana}, Section 5),
then shows that the Cartan subalgebra $\gh_{\gk}$ of $\gk$
decomposes as a direct sum
\begin{equation}\label{deco-kk}
\gh_{\gk} = (\gh_{\gk} \cap\bar{\gh}_{\gk}) \oplus \mathcal S
\oplus {\mathcal W}
\end{equation}
where $\mathcal W\subset \gh_{\gk}$ is any complementary subspace
of $(\gh_{\gk}\cap\bar{\gh}_{\gk})\oplus \mathcal S.$ Chose $\epsilon_{0}\in \gh_{\gk}^{*}\otimes \bar{\gh}_{\gk}^{*}$
such that
$$
\epsilon_{0} (\mathcal S , \cdot ) = \epsilon_{0} (\cdot ,
\bar{\mathcal S}) =0
$$
(i.e. (\ref{e3}) is satisfied) and $g_{\Delta} = \mathrm{Im}(
\epsilon\vert_{(\gh_{\gk} \cap \bar{\gh}_{\gk})^{\sigma}})$ is
non-degenerate. With this choice, $(\gk , \mathcal D^{0}, \epsilon
)$ is a symmetric $\gg$-admissible triple and the associated pair
$(\mathcal J, D^{0})$ has the property that  $J^{\mathcal J,
D^{0}}$ is integrable. }
\end{rem}

In order to study the non-degeneracy of $g_{\Delta}$ in general
(i.e. when $R_{0}$ is $\sigma$-parabolic, not necessarily
$\sigma$-positive) we chose a preferred basis of $\Delta$ and we
compute $g_{\Delta}$ in this basis. To simplify the arguments, we
assume that $R_{0}\cap \sigma (R_{0})$ is symmetric (this is
always satisfied, when $\gg$ is of inner type). Then $\gk \cap \bar{\gk}$ is reductive.  
Its real form $\Delta = (\gk \cap \bar{\gk})^{\sigma}$ is given by
$$
\Delta = \gh_{\gk}\cap \gh_{\gg} +
\sum_{\alpha\in R_{0}\cap \sigma (R_{0})} \mathbb{R}A_{\alpha}
+\sum_{\alpha \in R_{0}\cap \sigma (R_{0})} \mathbb{R} B_{\alpha},
$$
where, as in the proof of Lemma \ref{connection},  $A_{\alpha}:= E_{\alpha} -a_{\alpha} E_{\sigma (\alpha )}$ 
and $B_{\alpha}:= i( E_{\alpha} +
a_{\alpha} E_{\sigma (\alpha )})$. 
Since $R_{0}\cap \sigma (R_{0})$ is symmetric, $H_{\alpha} = [E_{\alpha}, E_{-\alpha }]\in \gk \cap \bar{\gk}$, for
any $\alpha \in R_{0}\cap\sigma (R_{0}).$ Define new vectors
$$
F^{+}_{\alpha}:= H_{\alpha} + H_{\sigma (\alpha )},\
F^{-}_{\alpha}:= i( H_{\alpha}  - H_{\sigma (\alpha )}),\quad \forall
\alpha \in R_{0}\cap \sigma (R_{0}).
$$
They belong to $\gh_{\gk}\cap \gh_{\gg}$. It follows that
$$
\gh_{\gk}\cap \gh_{\gg}=  \mathrm{Span}_{\mathbb{R}}\{
F^{+}_{\alpha},\ \alpha \in R_{0}\cap \sigma (R_{0})\} \oplus
\mathrm{Span}_{\mathbb{R}}\{ F^{-}_{\alpha}, \alpha\in R_{0}\cap
\sigma (R_{0})\} \oplus\mathcal C,
$$
where
$$
\mathcal C = \mathrm{Ann} (R_{0}\cap \sigma
(R_{0}))\vert_{\gh_{\gk}\cap \gh_{\gg}}.
$$
Let $\{ c_{1},\cdots , c_{s}\}$ be a basis of $\mathcal C$. Chose
a maximal system of linear independent vectors $\{ F_{{1}}^{+},
\cdots , F_{{p}}^{+}\}$ from $\{ F^{+}_{\alpha}, \alpha \in
R_{0}\cap \sigma (R_{0})\}$ and similarly, a maximal system of
linearly independent vectors $\{ F_{{1}}^{-}, \cdots ,
F_{{q}}^{-}\}$ from $\{ F^{-}_{\alpha}, \alpha \in R_{0}\cap
\sigma (R_{0})\} .$ It follows that the system of vectors
$$
\mathcal B := \{ c_{k},\ F^{+}_{{r}},\
F^{-}_{{t}},\
A_{\alpha},\ B_{\alpha}, \ \alpha \in R_{0}\cap \sigma (R_{0})\}
$$
(where $1\leq k\leq s$, $1\leq r\leq p$, $1\leq t\leq q$) form a
basis $\mathcal B$ of $\Delta .$

\begin{lem}\label{conditie}
Let $\epsilon \in \gk^{*}\otimes \bar{\gk}^{*}$ be given by
(\ref{exp}), such that condition (\ref{e3}) is satisfied. Assume, moreover, that
$R_{0}\cap \sigma (R_{0})$ is symmetric. With respect to the basis
$\mathcal B$ above, all the entries of
$g_{\Delta}=\mathrm{Im}(\epsilon\vert_{\Delta})$ are zero except:
\begin{align*}
g_{\Delta}(A_{\alpha}, B_{\beta}) &= -a_{\alpha} N_{\sigma
(\alpha)\beta} (\mu_{\alpha +\sigma ( \beta )}+a_{\sigma (\alpha )
+\beta} \mu_{\sigma (\alpha ) +\beta})\\
&+ N_{\alpha  \beta}(\mu_{\sigma (\alpha +\beta )} + a_{\alpha
+\beta }\mu_{\alpha +\beta})\\
g_{\Delta}( F^{+}_{r}, B_{\alpha})  &= (\mu_{\sigma (\alpha )}
+a_{\alpha}\mu_{\alpha})\alpha (F^{+}_{r})\\
g_{\Delta}( F^{-}_{t}, A_{\alpha}) & = i (\mu_{\sigma (\alpha
)} +a_{\alpha}\mu_{\alpha})\alpha (F^{-}_{t}).
\end{align*}
(We used the convention $N_{\delta\gamma }=\mu_{\delta +\gamma}=0$ for
$\delta , \gamma \in R_{0}$, such that $\delta +\gamma \notin R$). 
In particular, if $g_{\Delta}$ is non-degenerate, then
\begin{equation}\label{span}
\mathrm{dim}_{\mathbb{R}}\langle \alpha +\sigma (\alpha ), \
\alpha \in R_{0}\cap \sigma (R_{0})\rangle =
\mathrm{dim}_{\mathbb{R}} \langle \alpha -\sigma (\alpha ), \
\alpha \in R_{0}\cap \sigma (R_{0})\rangle .
\end{equation}

\end{lem}

\begin{proof}
The entries of $g_{\Delta}$ as above can be checked easily from
(\ref{exp}) and (\ref{e3}) and we omit the details
(for example, (\ref{e3}) means that
$F^{+}_{r}$ and $F^{-}_{t}$ belong to the kernel of
$g_{\Delta}\vert_{\gh_{\gk}\cap \gh_{\gg}}$). It is also easy to
check that if the matrix which represents $g_{\Delta}$ in the
basis $\mathcal B$ is non-degenerate, then $p=q$, i.e. relation
(\ref{span}) is satisfied. 
\end{proof}

\begin{rem}{\rm We now comment on condition (\ref{span}) from Lemma  \ref{conditie}.
Let $R_{0}\subset R$ be a closed subset of roots, such that
$R_{0}^{\prime}:= R_{0}\cap \sigma (R_{0})$ is symmetric and
(\ref{span}) holds. Since $R_{0}^{\prime}$ is 
symmetric and closed, it is the root system of the $\sigma$-invariant 
semisimple complex subalgebra 
$$
(\gg^{\prime})^{\mathbb{C}} := \gh^{\prime}+ \sum_{\alpha\in
R_{0}^{\prime}}\gg_{\alpha}\subset\gg^{\mathbb{C}},
$$
where $\gh^{\prime} = \mathrm{Span}_{\mathbb{C}} \{ H_{\alpha},\
\alpha\in R_{0}^{\prime}\}$ is a $\sigma$-invariant Cartan subalgebra of 
$(\gg^{\prime})^{\mathbb{C}}$. 
The action of $\sigma$ on the subset of roots $R_{0}^{\prime}\subset R$
is induced by an antilinear involution of  $(\gg^{\prime})^{\mathbb{C}}$, namely
by the restriction $\sigma^{\prime}$ of $\sigma$ to $(\gg^{\prime})^{\mathbb{C}}.$ 
Let $\gg^{\prime} = (\gg^{\prime})^{\mathbb{C}}\cap \gg$ be the real form of  
$(\gg^{\prime})^{\mathbb{C}} $ defined by $\sigma^{\prime}.$ Then
$\gh^{\prime}_{\gg^{\prime}}:= (\gh^{\prime})^{\sigma^{\prime}}$ is a maximally compact
Cartan subalgebra of $\gg^{\prime}.$  If we assume, in addition, that $R^{\prime}_{0}$ is irreducible, then
$(\gg^{\prime})^{\mathbb{C}}$ is a simple Lie algebra. It is easy to see 
that condition (\ref{span}) holds if and only if the automorphism of the Vogan diagram 
of $\gg^{\prime}$ has no fixed points. 
By inspecting the Vogan diagrams of
simple, non-complex real Lie algebras (see e.g. \cite{knapp},
Appendix C)  we deduce that  
(\ref{span}) holds if and only if 
$(\gg^{\prime})^{\mathbb{C}} $ is isomorphic to $\gsl
(2n+1, \mathbb{C})$ and $\gg^{\prime}$ is
the real form $\gsl (2n+1, \mathbb{R})\subset \gsl (2n+1,
\mathbb{C})$.}

\end{rem}

\subsubsection{Symmetric $\gg$-admissible triples of inner type}

Theorem \ref{main-applic} provides a complete explicit
description of symmetric $\gg$-admissible triples $(\gk , \mathcal D^{0},
\epsilon )$ of inner type, as follows.

\begin{thm}\label{inner-descr}
Let $\gg$ be a real form of inner type of $\gg^{\mathbb{C}}$, given by (\ref{alg}) (with 
$\sigma\vert_{R}= -\mathrm{Id}$). 
Consider a triple  $(\gk , \mathcal D^{0}, \epsilon )$ 
with $\gk$ the regular subalgebra  (\ref{g-k}), 
$\mathcal D^{0}$ as in Lemma \ref{connection}
and  $\epsilon \in\gk^{*}\otimes \bar{\gk}^{*}$
skew-Hermitian. Then   $(\gk , \mathcal D^{0}, \epsilon )$  is a
(symmetric) $\gg$-admissible triple (and 
the associated pair $(\mathcal J , D^{0})$ defines a complex structure 
$J^{\mathcal J , D^{0}}$ on $T^{*}G$) if and only if:

i) the root system $R_{0}$ of $\gk$ is a positive root system ($R_{0}=R^{+}$)  
and the Cartan part satisfies $\gh_{\gk}+ \bar{\gh}_{\gk}= \gh$;

ii) $\epsilon $ is of the form
$$
\epsilon = \epsilon_{0} + \sum_{\alpha\in R^{+}} \mu_{\alpha} (
\alpha \otimes \omega_{-\alpha }- a_{\alpha}
\omega_{\alpha}\otimes\alpha  )+\sum_{\alpha , \beta , \alpha +\beta\in
R^{+}}a_{\alpha}\mu_{\alpha +\beta}N_{\alpha \beta }
\omega_{\alpha}\otimes \omega_{-\beta }
$$
where $\epsilon_{0} \in \Lambda^{2} (\gh_{\gk})$ is trivially extended to $\gk$, 
and  $\mu_{\alpha}$ ($\alpha \in R^{+}$)   are arbitrary real constants;

iii) $\mathrm{Im}(\epsilon\vert_{\gh_{\gk}\cap
i\gh_{\mathbb{R}}})$ is non-degenerate.
\end{thm}

\begin{proof}
We use Theorem \ref{main-applic}. Since $\sigma\vert_{R}=-\mathrm{Id}$,
$R_{0}\cap \sigma (R_{0})$ is symmetric and relation (\ref{span})
implies that $R_{0} \cap (-R_{0}) = \emptyset$. Since $R_{0}\cup
(-R_{0}) = R$, from  a result of Bourbaki we obtain that $R_{0}=
R^{+}$ is a positive root system. Condition (\ref{assumption-not})
is satisfied (this follows from
$\gh_{\gk} + \bar{\gh}_{\gk}= \gh$ and $\sigma \vert_{R}=
-\mathrm{Id}$). Conditions (\ref{suplimentara}) and (\ref{e3}) do
not apply $(R^{+}$ is skew-symmetric) and the intersection $\gk
\cap \bar{\gk}$ reduces to its Cartan part $\gh_{\gk}\cap
\bar{\gh}_{\gk}$.
\end{proof}

\section{Special complex geometry}\label{special-geom}

In this section we develop further applications of Theorem
\ref{main-thm}, in relation to special complex geometry.

\begin{prop}\label{ex-1} Let $(M,J, D)$ be a manifold
with an almost complex structure $J$ and a linear connection $D$.
The almost complex structure $J^{\pm}$ on $T^{*}M$, defined by $D$
and the generalized  complex structure
$$
{\mathcal J}^{\pm } := \left(\begin{tabular}{cc} $J$ & $0$\\
$0$ & $\pm J^{*}$
\end{tabular}\right)
$$
is integrable if and only if $J$ is a complex structure, $D_{X}(J)
= \pm J D_{JX} (J)$ and
\begin{equation}\label{curv-part}
(R^{D}_{X, Y} - R^{D}_{JX, JY} ) (Z) \pm ( R^{D}_{JX, Y} +
R^{D}_{X, JY}) (JZ) =0,
\end{equation}
for any $X, Y, Z\in TM.$

\end{prop}

\begin{proof} The generalized complex structure $\mathcal J^{+}$ is symmetric, with
holomorphic bundle $T^{1,0}M\oplus \mathrm{Ann} (T^{0,1}M)$, while
$\mathcal J^{-}$ is skew-symmetric, with holomorphic bundle
$T^{1,0}M\oplus \mathrm{Ann}(T^{1,0}M)$. From Theorem
\ref{main-thm}, if ${\mathcal J}^{\pm }$ is integrable, then
the bundle $T^{1,0}M$ is involutive, i.e $J$ is an (integrable) complex structure. Also,
$D_{\Gamma (T^{1,0}M)}\Gamma (T^{1,0}M) \subset \Gamma (T^{1,0}M)$
if and only if $D_{X}(J) = - J D_{JX} (J)$, while $D_{\Gamma
(T^{1,0}M)}\Gamma (T^{0,1}M) \subset \Gamma (T^{0,1}M)$ if and
only if $D_{X}(J) =  J D_{JX} (J)$, for any $X\in TM.$ The
condition $R^{D}\vert_{T^{1,0}M, T^{1,0}M}(\tau (T^{1,0}M))=0$
from Theorem \ref{main-thm} translates to (\ref{curv-part}).
Condition (\ref{ec}) from Theorem \ref{main-thm} is also
satisfied, because $\alpha =0$ (in both cases). Our claim follows.

\end{proof}

As already mentioned in the introduction, the first statement of
the following corollary was proved in \cite{alek} using different
methods.

\begin{cor}\label{cor-1} Consider the setting of Proposition \ref{ex-1}.

i) If $(J, D)$ is a special complex structure, i.e. $J$ is
integrable and $D$ is flat, torsion-free, such that
$$
(d^{D}J)_{X,Y}:= D_{X}(J)(Y) - D_{Y}(J)(X)=0, \quad \forall X,Y\in
TM,
$$
then $J^{+}$ is integrable.

ii) If $D=D^{g}$ is the Levi-Civita connection of an almost
Hermitian structure $(g, J)$, then $J^{+}$ is integrable if and
only if $(g, J)$ is K\"{a}hler and $J^{-}$ is integrable if and
only if $J$ is integrable and the curvature of $g$ satisfies
$$
( R^{D}_{X,Y} - R^{D}_{JX, JY} ) (Z) - (R^{D}_{JX, Y} + R^{D}_{X,
JY}) (JZ) =0,\quad \forall X, Y, Z\in TM.
$$

iii) If $D$ is the Chern connection of a Hermitian structure $(J,
g)$, then both $J^{\pm}$ are integrable.

\end{cor}

\begin{proof} The claims follow from Proposition \ref{ex-1}.
For {\it i)}, we remark that the special complex condition
$d^{D}J=0$ implies $D_{X}(J) = J D_{JX} (J)$ for any $X\in TM.$
For {\it ii)} we use that $D^{g}_{X}(J) = - JD^{g}_{JX}(J)$, for
any $X\in TM$, if and only if $J$ is integrable (see \cite{GH} or
Proposition 1 of \cite{gaud}). This proves the statement for
$J^{-}.$ The statement for $J^{+}$ follows as well: if $J$ is
integrable and $D^{g}_{X}(J) =  JD^{g}_{JX}(J)$, then $D^{g}J=0$
and $(g, J)$ is K\"{a}hler. For {\it iii)} we use that the Chern
connection is  Hermitian with curvature of type $(1,1).$
\end{proof}

The following lemma is a mild improvement of Lemma 6 of
\cite{alek}.

\begin{lem}
Let $(M,\omega , D)$ be a manifold with a non-degenerate $2$-form
$\omega$ and a linear connection $D$. The almost complex structure
on $T^{*}M$ defined by $D$ and the (skew-symmetric) generalized
complex structure
$$
{\mathcal J}^{{\omega}}= \left(\begin{tabular}{cc}
$0$ & $\omega^{-1}$\\
$-\omega$ & $0$\\
\end{tabular}\right)
$$
is integrable if and only if $D$ is flat and, for any $X,Y, Z\in
{\mathcal X} (M)$,
\begin{align*}
(d\omega ) (X,Y, Z)  - (D_{Z}\omega ) (X,Y) - \omega
(T^{D}_{Z}X,Y) -\omega (X,T^{D}_{Z}Y)=0.
\end{align*}

\end{lem}

\begin{proof} The holomorphic bundle of ${\mathcal J}^{\omega}$ is
$L(T^{\mathbb{C}}M, i\omega^{\mathbb{C}} )$ and the claim follows
from Theorem \ref{main-thm} and relation (\ref{simple-ec}).

\end{proof}

{\bf Author's address:} Permanent: Institute of Mathematics 'Simion Stoilow' of the
Romanian Academy, Research Unit 4, Calea Grivitei nr. 21,
Bucharest, Romania; liana.david@imar.ro\\

Present: Lehrstuhl f\"{u}r Mathematik VI, Institut f\"{u}r Mathematik,
Universit\"{a}t  Mannheim, 
A5, 6, 68131, Mannheim, Germany;
ldavid@mail.uni-mannheim.de

\end{document}